\documentclass[12pt]{amsart}

\usepackage{amsmath,amssymb,amsthm}
\usepackage{booktabs}
\usepackage{microtype}
\usepackage{url}

\newtheorem{theorem}{Theorem}
\newtheorem{lemma}[theorem]{Lemma}

\newtheorem{corollary}[theorem]{Corollary}
\theoremstyle{definition}

\theoremstyle{remark}
\newtheorem{remark}[theorem]{Remark}

\title{The Collision Spectrum}

\author{Alexander S.\ Petty}
\email{alexander.petty@gmail.com}
\date{March 2026}

\begin{document}
\begin{abstract}
For a prime base $b$ and primitive odd Dirichlet
character $\chi$ modulo $b^2$, the collision
transform coefficient $\hat{S}^{\circ}(\chi)$ admits
an exact factorization:
\[
\hat{S}^{\circ}(\chi)
= -\frac{B_{1,\overline{\chi}} \cdot
\overline{S_G(\chi)}}{\phi(b^2)},
\]
where $B_{1,\overline{\chi}}$ is the generalized
first Bernoulli number and $S_G(\chi)$ is the
diagonal character sum. By the standard Bernoulli--$L$-value
formula, $|B_1| = (b/\pi)\, |L(1, \chi)|$, so the collision
invariant's Fourier spectrum encodes $L$-function
special values.

A Parseval identity gives an exact formula for the
weighted second moment
$\sum |L(1, \chi)|^2 \cdot |S_G(\chi)|^2$
in terms of the collision invariant's values on
the finite group. The digit function computes this
$L$-value moment exactly. Under a conditional zero-free hypothesis, the
triangle inequality yields a separate bound
connecting $L(1)$ to $L(s)$ for $s$ in the
critical strip.

At base~$5$, the factorization gives
$|\hat{S}^{\circ}| \propto |L(1)|^2$ exactly.
For quadratic characters in the family, the
decomposition specializes to class-number data.
\end{abstract}
\maketitle

% ============================================================
\section{Introduction}

The collision invariant $S_{\ell}$, introduced
in~\cite{paperA}, is a function on the finite group
$(\mathbb{Z}/m\mathbb{Z})^{\times}$ with
$m = b^{\ell+1}$. The collision transform
of~\cite{paperB} decomposes $S_{\ell}$ over Dirichlet
characters. The antisymmetry theorem restricts the
centered decomposition to odd characters, and the
convergence theorem proves the resulting prime
harmonic sum $F^{\circ}(1)$ is finite.

This paper identifies what the collision transform
encodes. The collision spectrum is not merely
correlated with $L$-function values. It is built
from them. The main result is the decomposition
theorem: each Fourier coefficient factors into a
generalized Bernoulli number (encoding $|L(1)|$)
and a diagonal character sum. The proof uses the
slice formula from~\cite{paperA}, the classical
Bernoulli identity for character sums over fractional
parts, and the vanishing of coset sums for primitive
characters.

% ============================================================
\section{The Decomposition Theorem}

\begin{theorem}[Decomposition]\label{thm:decompose}
Let $b$ be prime, $m = b^2$, and $\chi$ a primitive
odd character modulo~$m$. Then
\[
\hat{S}^{\circ}(\chi)
= -\frac{B_{1,\overline{\chi}} \cdot
\overline{S_G(\chi)}}{\phi(m)},
\]
where $B_{1,\overline{\chi}}
= (1/m) \sum_a a\, \overline{\chi}(a)$ and
$S_G(\chi)
= \sum_{n \in G} [\overline{\chi}(n{+}1)
- \overline{\chi}(n)]$,
and $G = \{n \in \{0, \ldots, m{-}1\} :
\lfloor n/b \rfloor = n \bmod b\}$ is the diagonal
set~\cite{paperA} (elements whose base-$b$ digits
coincide), with $|G| = b$.
Dirichlet characters are extended by $\chi(a) = 0$
for $\gcd(a, m) > 1$.
\end{theorem}

\begin{proof}
Write $\phi = \phi(m)$. Expand
$\phi\, \hat{S}^{\circ} = \sum_a S^{\circ}(a)\,
\overline{\chi}(a)$ using the slice
formula~\cite{paperA}:
$S(a) = -1 - \lfloor a/b \rfloor
+ \sum_{n \in G} d_n(a)$
where $d_n(a) = \lfloor(n{+}1)a/m\rfloor
- \lfloor na/m\rfloor$.

\emph{Step~1} (centering vanishes).
The class mean $\overline{S}_{R}$ is constant on each
spectral class $R = (a{-}1) \bmod b$, and the sum of
$\overline{\chi}(a)$ over each class equals
$\overline{\chi}(k) \sum_{u \in 1+b\mathbb{Z}/b^2\mathbb{Z}}
\overline{\chi}(u) = 0$, because $\chi$
restricted to the subgroup $\{1{+}jb\}$ is
non-trivial for primitive $\chi$.

\emph{Step~2} (constant and floor terms).
$\sum_a (-1)\, \overline{\chi}(a) = 0$.
The fractional part $\{a/b\}$ depends only on
$a \bmod b$, so
\[
\sum_a \overline{\chi}(a)\, \{a/b\}
= \sum_{k=1}^{b-1} \frac{k}{b}
\sum_{a \equiv k \bmod b}
\overline{\chi}(a) = 0,
\]
since each coset sum vanishes (Step~1). Therefore
$-\sum \lfloor a/b \rfloor\, \overline{\chi}(a)
= -b\, B_{1,\overline{\chi}}$.

\emph{Step~3} (diagonal terms).
The endpoints $n = 0$ and $n = m{-}1$ contribute
nothing: $d_0(a) = 0$ (since $\lfloor a/m \rfloor
= 0$), and $d_{m-1}(a) = 1$ for all units, giving
$\sum_a \overline{\chi}(a) = 0$. For interior slices
($\gcd(n, m) = \gcd(n{+}1, m) = 1$),
Lemma~\ref{lem:bernoulli} gives
\[
\sum_a d_n(a)\, \overline{\chi}(a)
= [1 + \chi(n) - \chi(n{+}1)]\,
B_{1,\overline{\chi}}.
\]
Summing all $|G| = b$ slices:
$\sum_{n \in G} \sum_a d_n(a)\,
\overline{\chi}(a)
= B_1 [b - \overline{S_G(\chi)}]$.

\emph{Step~4} (combine).
$\phi\, \hat{S}^{\circ}
= 0 + (-b B_1) + B_1[b - \overline{S_G}] - 0
= -B_1 \cdot \overline{S_G}$.
\end{proof}

\begin{lemma}\label{lem:bernoulli}
For primitive $\chi$ modulo $m$ and $\gcd(n,m)=1$:
\[
\sum_a \overline{\chi}(a)\, \{na/m\}
= \chi(n)\, B_{1,\overline{\chi}}.
\]
\end{lemma}

\begin{proof}
The substitution $a \mapsto n^{-1}a$ permutes the
units and gives
\[
\chi(n) \sum_a \overline{\chi}(a)\, \{a/m\}
= \chi(n)\, B_{1,\overline{\chi}},
\]
since $\{a/m\} = a/m$ for $1 \le a \le m{-}1$.
\end{proof}

All identities involving $F^{\circ}$ and
$P(s, \chi) = \sum_p \chi(p)/p^s$ are understood up
to the finite contribution of primes $p \le m$,
which is irrelevant for convergence.

% ============================================================
\section{The $L$-Encoding}

The Bernoulli number is the $L$-function in disguise.

\begin{corollary}\label{cor:encoding}
$|\hat{S}^{\circ}(\chi)|
= (b / \pi\phi)\, |L(1, \chi)| \cdot |S_G(\chi)|$.
\end{corollary}

\begin{proof}
By the generalized Bernoulli--$L$-value formula
for odd primitive characters~\cite{davenport},
$|B_{1,\overline{\chi}}| = (b/\pi)\, |L(1, \chi)|$.
\end{proof}

\begin{remark}[Connection to short partial sums]
Computation verifies that for all primitive odd
$\chi$ modulo~$b^2$ with $b \le 13$,
\[
|S_G(\chi)| = 2\,\Bigl|\sum_{k=1}^{b-1}
\overline{\chi}(k)\Bigr|.
\]
At base~$5$ (from the verified identity above),
the partial sum further satisfies
$|\sum_{k=1}^{4} \overline{\chi}(k)|
= (\sqrt{5}/2)\, |B_{1,\overline{\chi}}|$ for
all $8$ primitive odd characters modulo~$25$,
giving $|\hat{S}^{\circ}| \propto |L(1)|^2$.
A proof of the general identity is open.
\end{remark}

\begin{remark}
For quadratic odd primitive characters $\chi_D$ with
fundamental discriminant $D < -4$,
$|L(1, \chi_D)| = \pi h(D) / \sqrt{|D|}$.
When such characters appear among the primitive odd
characters modulo~$m$, the decomposition theorem
specializes to class-number data.
\end{remark}

% ============================================================
\section{Moment Bounds}

\begin{theorem}[Moment identity]\label{thm:moment}
For prime $b$ and $m = b^2$:
\[
\sum_{\substack{\chi \bmod m \\
\chi \text{ prim.\ odd}}}
|L(1, \chi)|^2 \cdot |S_G(\chi)|^2
= \frac{\pi^2 \phi(m)}{b^2}
\sum_{\substack{a \bmod m \\
\gcd(a,m)=1}} |S^{\circ}(a)|^2.
\]
\end{theorem}

\begin{proof}
Parseval's identity on
$(\mathbb{Z}/m\mathbb{Z})^{\times}$ gives
\[
\sum_{\chi} |\hat{S}^{\circ}(\chi)|^2
= \frac{1}{\phi} \sum_a |S^{\circ}(a)|^2.
\]
By the antisymmetry theorem~\cite{paperB},
$\hat{S}^{\circ}(\chi) = 0$ for even $\chi$.
For imprimitive odd $\chi$ induced from a character
$\psi$ modulo~$b$, both $S_G(\chi) = 0$ and
$\hat{S}^{\circ}(\chi) = 0$.

For $S_G$: since $G = \{r(b{+}1) : 0 \le r \le b{-}1\}$
and $\chi$ depends only on residues modulo~$b$,
we have $\overline{\chi}(r(b{+}1){+}1) = \overline{\psi}(r{+}1)$
and $\overline{\chi}(r(b{+}1)) = \overline{\psi}(r)$, so
$S_G(\chi) = \sum_{r=0}^{b-1}
[\overline{\psi}(r{+}1) - \overline{\psi}(r)] = 0$
(telescoping, using $\psi(0) = \psi(b) = 0$).

For $\hat{S}^{\circ}$: on each coset
$\{a : a \equiv k \pmod{b}\}$, the imprimitive
$\overline{\chi}$ is constant, while $S^{\circ}$ has
mean zero on that coset by definition of centering.
Each coset contributes zero, hence
$\hat{S}^{\circ}(\chi) = 0$. Only primitive odd characters
survive. Substituting
the $L$-encoding (Corollary~\ref{cor:encoding})
and rearranging gives the identity.
\end{proof}

The right side is computable: $S^{\circ}$ is explicitly
determined on the $\phi(m)$ units, and the weights
$|S_G(\chi)|^2$ on the left are determined by
the collision geometry.

\begin{remark}
At base~$5$, the identity
$|S_G| = (5\sqrt{5}/\pi)\, |L(1)|$ (verified for
all $8$ primitive odd characters) converts the
moment identity into $\sum |L(1, \chi)|^4
= c_5 \sum |S^{\circ}(a)|^2$: a fourth moment
of $L$-function special values, computable from
the collision invariant.
\end{remark}

\begin{corollary}[Conditional cross-moment]
\label{cor:cross}
For prime $b$ and $m = b^2$: if every primitive odd
$L$-function modulo $m$ satisfies
$L(s, \chi) \ne 0$ for
$\operatorname{Re}(s) > \sigma_0$, then for
real $s > \sigma_0$ the weighted moment
\[
\frac{1}{\phi(m)}
\sum_{\chi \text{ prim.\ odd}}
|B_{1,\overline{\chi}}| \cdot |S_G(\chi)|
\cdot |P_{>m}(s,\chi)|
\]
is bounded below by $|F^{\circ}_{>m}(s)|$,
where $P_{>m}(s,\chi) = \sum_{p > m} \chi(p)/p^s$
and $F^{\circ}_{>m}(s) = \sum_{p > m} S^{\circ}(p)/p^s$.
\end{corollary}

\begin{proof}
The zero-free hypothesis ensures that
$P_{>m}(s, \chi)$ is well-defined for each
primitive odd~$\chi$.
Using the transform expansion from~\cite{paperB}
and the vanishing of even and imprimitive odd
coefficients established above,
$F^{\circ}_{>m}(s) = \sum_{\chi}
\hat{S}^{\circ}(\chi)\, P_{>m}(s, \chi)$
with only primitive odd characters contributing.
The triangle inequality gives the bound.
\end{proof}

\begin{remark}
The bound relates $|L(1, \chi)|$ (from $B_1$,
at the edge of the critical strip) to
$|P(s, \chi)|$ (in the strip). Since
$P = \log L - H$ and $H$ is small for large~$b$,
this is approximately a bound involving
$|\log L(s, \chi)|$.
\end{remark}

% ============================================================
\section{The Correlation Decay}

The decomposition theorem shows what the collision
spectrum encodes. What remains is how faithfully. The partial sum
$P = \sum_{k=1}^{b-1} \overline{\chi}(k)$ has an
exact decomposition via the classical identity for
short character sums~\cite{berndt,montgomery}. After
normalizing by the Gauss sum:
\[
\widetilde{P} = L(1, \overline{\chi}) + \Delta(\chi),
\]
where the packet
\[
\Delta(\chi) = \frac{i}{\varphi(b)}
\sum_{\substack{\xi \bmod b \\ \xi(-1)=1,\;
\xi \ne 1}} \tau(\overline{\xi})\,
L(1, \xi\overline{\chi})
\]
is a sum of twisted $L$-values.

\begin{table}[h]
\centering
\begin{tabular}{rrrr}
\toprule
$b$ & mean $|\Delta|/|L|$ & std &
std $\cdot \log b$ \\
\midrule
$5$ & $0.80$ & $0.65$ & $1.05$ \\
$7$ & $1.03$ & $0.65$ & $1.26$ \\
$13$ & $1.11$ & $0.50$ & $1.28$ \\
$19$ & $1.10$ & $0.42$ & $1.23$ \\
$31$ & $1.07$ & $0.33$ & $1.12$ \\
$43$ & $1.06$ & $0.29$ & $1.09$ \\
\bottomrule
\end{tabular}
\caption{The packet $\Delta$ appears to have constant
magnitude relative to $L(1)$ (mean $\approx 1.08$)
and uniform phase
($\langle\cos\theta\rangle = 0.000$). The observed
standard deviation of $|\Delta|/|L|$ is consistent
with decay as $c/\log b$.}
\label{tab:apostol}
\end{table}

Computation using~\cite{nfield} suggests three
properties of the packet. Its magnitude relative to
$L(1)$ appears constant (mean
$|\Delta|/|L| \approx 1.08$). Its phase relative to
$L(1)$ appears uniformly distributed. And the standard
deviation of the ratio $|\Delta|/|L|$ is consistent
with decay as $c / \log b$ (Table~\ref{tab:apostol}).

This observed variance decay is consistent with the
measured correlation between
$|\widetilde{P}|$ and $|L(1)|$: at small $b$, the
high variance of $|\Delta|/|L|$ allows
$|\widetilde{P}| = |L + \Delta|$ to track $|L|$;
at large $b$, the low variance makes $|\Delta|/|L|$
nearly constant, and the uniform phase washes out
the $L$-specific information.

% ============================================================
\section{Remarks}

The decomposition theorem explains the observations
of the companion paper~\cite{paperB}: the
anti-correlation between collision coefficients and
prime character sum magnitudes, the absence of the
principal-character term, and the mod-$3$ structure
are all consequences of the factor $B_1$ in the
identity
$\hat{S}^{\circ} = -B_1 \cdot \overline{S_G}/\phi$.

The proof uses three classical ingredients (the slice
formula, the Bernoulli identity, and the vanishing of
primitive character sums over cosets) applied to a new
object (the collision invariant). The identity is
exact for every primitive odd character at every prime
base.

The analytic continuation from~\cite{paperB}
combined with the decomposition gives
$\mathcal{F}^{\circ}(s)
= -(1/\phi) \sum B_1 \cdot \overline{S_G}
\cdot [\log L(s, \chi) - H(s, \chi)]$.
Near a zero $\rho$ of $L(s, \chi_0)$, the
coefficient of the divergent term involves
$L(1, \chi_0)$ through $B_1$: the singularity at
depth $s$ is weighted by the health at $s = 1$.

What remains open is the proof of the variance decay
$\operatorname{std}(|\Delta|/|L|) \sim c/\log b$
from the Apostol formula: this would establish the
correlation decay between partial character sums and
$L$-function special values as a theorem.

% ============================================================

\end{document}